\documentclass{amsart}

\usepackage{amsmath}
\usepackage{amssymb}
\usepackage{rotating}
\usepackage{graphicx}
\usepackage[tableposition=below]{caption}
\usepackage{bm}
\PassOptionsToPackage{colorlinks,allcolors=black}{hyperref}
\usepackage{hyperref}

\DeclareGraphicsExtensions{.png,.pdf,.eps}
\usepackage[all,2cell,cmtip]{xy}
\usepackage{tikz}
\usepackage{color}
\usepackage{longtable}
\usepackage{soul}
\usepackage{float}

\newtheorem{theorem}{Theorem}[section]
\newtheorem{lemma}[theorem]{Lemma}

\newtheorem{proposition}[theorem]{Proposition}
\newtheorem{conjecture}[theorem]{Conjecture}

\theoremstyle{definition}

\newtheorem{definition}[theorem]{Definition}

\newtheorem{remark}[theorem]{Remark}

\newtheorem{example}[theorem]{Example}

\def\P{{\mathbb P}}

\def\cE{{\mathcal E}}

\def\cM{{\mathcal M}}
\def\cN{{\mathcal{N}}}
\def\cO{{\mathcal{O}}}

\def\cS{{\mathcal S}}

\def\cOperatorname#1{\mathop{\rm #1}\nolimits}

\def\codim{\cOperatorname{codim}}

\def\deg{\cOperatorname{deg}}

\def\det{\cOperatorname{det}}

\def\ME{{\cOperatorname{ME}}}

\newcommand{\cME}[1]{\cOverline{\ME}}

\makeindex

\title{{Smooth Fano varieties with pseudoindex equal to half of their dimension}}

\author{Kiwamu Watanabe}
\date{\today}
\address{Department of Mathematics, Faculty of Science and Engineering, Chuo University.
1-13-27 Kasuga, Bunkyo-ku, Tokyo 112-8551, Japan}
\email{watanabe@math.chuo-u.ac.jp}
\thanks{The author is partially supported by JSPS KAKENHI Grant Number 21K03170 and 25K06940.}


\begin{document}

\maketitle

\begin{abstract} Let $X$ be a complex smooth Fano variety of dimension $n$. Assume that $X$ admits a birational contraction of an extremal ray. In this paper, we give a classification of such $X$ when the pseudoindex is equal to $\frac{\dim X}{2}$. 
\end{abstract}

\section{Introduction}
Let $X$ be a complex $n$-dimensional smooth projective variety. We call a variety $X$ {\it Fano} if the anticanonical divisor $-K_X$ is ample. Fano varieties (with mild singularities) are considered one of the three fundamental building blocks of algebraic varieties in birational geometry. Two important invariants arise in studying {smooth} Fano varieties: the largest integer that divides $-K_X$ in the Picard group of $X$, which is known as the {\it Fano index} of $X$. The second is the minimum anticanonical degree of rational curves on $X$, called the {\it pseudoindex} of $X$. The Fano index of $X$ and the pseudoindex of $X$ are denoted by ${ \rm i_X}$ and $\iota_X$, respectively.
Historically, {smooth} Fano varieties have been studied primarily from the perspective of the Fano index. It is known that ${ \rm i_X} \leq n+1$ and that {smooth} Fano varieties with ${ \rm i_X} \geq n-2$ can be classified (see, for example, \cite{Isk}). By contrast, addressing similar problems for the pseudoindex is considerably more challenging. For instance, it follows from Mori's bend-and-break lemma that $\iota_X \leq n+1$. Moreover, K. Cho, Y. Miyaoka, and N. I. Shepherd-Barron \cite{CMSB} proved that when $\iota_X = n+1$, $X$ is isomorphic to projective space $\P^n$. Additionally, T. Dedieu and A. H\"oring \cite{DH17} proved that when $\iota_X = n$, $X$ is isomorphic to a smooth quadric hypersurface $Q^n$. The author \cite{Wat24} classified {smooth} Fano varieties with Picard number $\rho_X > 1$ when $\iota_X = n-1$ and $\iota_X = n-2$. However, unlike the case of the Fano index, the classification of varieties with $\iota_X = n-1$ and $\iota_X = n-2$ is still unresolved when $\rho_X = 1$.
As a broader geometric problem concerning general {smooth} Fano varieties, we state the following generalized Mukai conjecture:

\begin{conjecture}[\cite{BCDD03}]\label{conj:GM} For an $n$-dimensional smooth Fano variety $X$, we have 
$$
\rho_X(\iota_X-1)\leq n.
$$ Furthermore, the equality holds if and only if $X$ is isomorphic to a product of $\rho_X$ projective spaces, $(\P^{\iota_X-1})^{\rho_X}$.
\end{conjecture}

S. Mukai's original conjecture \cite[Conjecture~4]{Mukai88}, now known as the Mukai conjecture, replaces the pseudoindex with the Fano index in the above statement. The generalized Mukai conjecture is partially motivated by the following theorem of J. Wi\'sniewski:

\begin{theorem}[{\cite{Wis90b}}]
\label{them:Wis}
Let $X$ be an $n$-dimensional smooth Fano variety. If $\iota_X > \frac{n}{2} + 1$, then $\rho_X = 1$. Moreover, if ${ \rm i_X} = \frac{n}{2} + 1$ and $\rho_X > 1$, then $X$ is isomorphic to a product of two projective spaces, $(\P^{{\rm i_X} - 1})^2$.
\end{theorem}

In the latter part of this theorem, G. Occhetta \cite[Theorem~1.1]{Occ06} extended the result by replacing the Fano index with the pseudoindex. For the case where ${ \rm i_X} = \frac{n+1}{2}$ and $\rho_X > 1$, the classification of $n$-dimensional {smooth} Fano varieties was obtained by Wi\'sniewski \cite{Wis91ind}. The extension of this classification from the Fano index to the pseudoindex was provided by the author \cite{Wat24M}. Additionally, in the next case, $n$-dimensional {smooth} Fano varieties with ${ \rm i_X} = \frac{n}{2}$ and $\rho_X > 1$ were classified in a series of papers by Wi\'sniewski et al. \cite{PSW92, Wis93, BW96, Wis94}.

This paper aims to study $n$-dimensional {smooth} Fano varieties $X$ with ${\iota}_X = \frac{n}{2}$ and $\rho_X > 1$ when $X$ admits a birational contraction of an extremal ray. {In the following, for a smooth projective variety $X$ and its smooth projective closed subvariety $Y$, we denote by $Bl_Y X$ the blow-up of $X$ along $Y$.}

\begin{theorem}\label{them:main} Let $X$ be an $n$-dimensional smooth Fano variety with $\iota_X=\frac{n}{2}\geq 2$. Assume that $X$ admits a birational contraction $\varphi: X\to Y$ of an extremal ray. Then $X$ is isomorphic to one of the following:
\begin{enumerate}
\item {$Bl_{\P^{\frac{n}{2} - 1}}(\P^n)$ with a linear subvariety $\P^{\frac{n}{2} - 1}$;} 
\item {$Bl_{\P^{{\frac{n}{2} - 1}}}(Q^n)$ with a linear subvariety $\P^{\frac{n}{2} - 1}$;}
\item {$Bl_{Q^{{\frac{n}{2} - 1}}}(Q^n)$ with a smooth quadric $Q^{{\frac{n}{2} - 1}}$ of $Q^n$ not contained in a linear subspace of $Q^n$;}
\item {$Bl_{\P^{\frac{n}{2} - 2}}(\P^n)$ with a linear subspace $\P^{\frac{n}{2} - 2}$;}
\item $\P(\cO_{\P^{\frac{n}{2}+1}}(2)\oplus\cO_{\P^{\frac{n}{2}+1}}^{\oplus \frac{n}{2}-1})$;
\item $\P(\cO_{Q^{\frac{n}{2}+1}}(1)\oplus\cO_{Q^{\frac{n}{2}+1}}^{\oplus \frac{n}{2}-1})$;
\item $\P(\cO_{\P^{\frac{n}{2}+1}}^{\oplus 2}(1)\oplus\cO_{\P^{\frac{n}{2}+1}}^{\oplus \frac{n}{2}-2})$;
\item $\P^1\times \P(\cO_{\P^2}(1)\oplus\cO_{\P^2})$.
\end{enumerate}
{If $X$ is as in {\rm (i)}, {\rm (iv)}, or {\rm (viii)}, then $ \rm i_X=1$; otherwise ${\rm i_X} = \iota_X$.}
\end{theorem}

This paper is structured as follows. Chapter 2 summarizes what is known about contractions and the deformation theory of rational curves. Furthermore, we verify through explicit calculations that the pseudoindex of the {smooth} Fano varieties appearing in Theorem~\ref{them:main} is $\frac{n}{2}$. 
In Chapter 3, we classify {smooth} Fano varieties with pseudoindex $\frac{n}{2}$ that exhibit special structures, such as blow-ups and projective bundles. In Chapter 4, we will classify {smooth} Fano varieties appearing in Theorem~\ref{them:main} into cases with small contraction and cases with divisorial contraction. {In particular, we show that any smooth Fano variety admitting either a small or a divisorial contraction also admits a structure as a blow-up or a projective bundle, thereby allowing us to reduce
the classification to the cases covered in Chapter 3.}

\subsection*{Notation and Conventions}\label{subsec:NC} In this paper, we work over the complex number field. We use standard notation and conventions as in the books \cite{Har}, \cite{Kb} and \cite{KM}.  
\begin{itemize}
\item We denote by $\P^n$ the projective space of dimension $n$ and by $Q^n$ a smooth quadric hypersurface of dimension $n$. 
\item A {\it curve} means a projective curve.
\item {{For a smooth projective variety $X$ and its smooth projective closed subvariety $Y$, we write $N_{Y/X}$ for the normal bundle of $Y$ in $X$.}} 
\item For projective varieties $X, Y$ and $F$, a smooth surjective morphism $f:X\to Y$ is called an {\it $F$-bundle} if any fiber of $f$ is isomorphic to $F$. 
\item {A contraction of an extremal ray is called simply an {\it elementary contraction} when no extremal ray is specified.}
\item For a smooth projective variety $X$, we denote by $\rho_X$ the Picard number of $X$.
\item For a smooth projective variety $X$, the corresponding numerical class of a curve $C\subset X$ is denoted by $[C]\in N_1(X)$. 
\item For a vector bundle $\cE$ over a variety $X$, we denote by $\cE^{\vee}$ the dual of $\cE$. 
\end{itemize}

\section{Preliminaries}

\subsection{Contractions} 

\begin{definition} 
For a smooth projective variety $X$ and a $K_X$-negative extremal ray $R \subset \overline{NE}(X)$, the \textit{length} of $R$ is defined as
$$\ell(R) := \min \{-K_X \cdot C \mid C \text{ is a rational curve and } [C] \in R \}.$$
\end{definition}

\begin{proposition}[{Ionescu-Wi\'sniewski inequality \cite[Theorem~0.4]{Ion86}, \cite[Theorem~1.1]{Wis91}}]\label{prop:Ion:Wis} 
Let $X$ be a smooth projective variety, and let $\varphi: X \to Y$ be the contraction of a $K_X$-negative extremal ray $R\subset \overline{NE}(X)$, where $E$ denotes {an irreducible component of} the exceptional locus of the contraction $\varphi$. Let $F$ be an irreducible component of a non-trivial fiber of $\varphi$ {contained in $E$}. Then,
\[
\dim E + \dim F \geq \dim X + \ell(R) - 1.
\]
\end{proposition}

\begin{theorem}[{\cite[Theorem~1.3]{HNov13}}]\label{them:HN13} 
Let $X$ be a smooth projective variety, and let $\varphi: X \to Y$ be the contraction of an extremal ray $R\subset \overline{NE}(X)$ such that the dimension of each fiber is $d$ and $\ell(R) = d+1$. Then, $\varphi$ is a projective bundle.
\end{theorem}

\begin{theorem}[{\cite[Theorem~5.1]{AO02}}]\label{them:AO02} 
For a smooth projective variety $X$, the following are equivalent:
\begin{enumerate}
    \item There exists an extremal ray $R\subset \overline{NE}(X)$ such that the associated contraction is divisorial, and the fibers are of dimension $\ell(R)$.
    \item $X$ is a blow-up of a smooth projective variety $X'$ along a smooth subvariety of codimension $\ell(R)+1$.
\end{enumerate}
\end{theorem}

In this paper, we often use the following inequality { concerning the dimensions of the fibers of a contraction} without mentioning it:

\begin{lemma}\label{lem:Serre} Let $X$ be a smooth projective variety. For projective subvarieties $Y$ and $Z$ of $X$, if $Y\cap Z$ is nonempty, then we have 
$$
\dim (Y\cap Z) \geq \dim Y+\dim Z-\dim X.
$$
In particular, when $X$ admits two distinct elementary contractions $\varphi: X\to Y$ and $\psi: X \to Z$, { if an irreducible component $F$ of a fiber of $\varphi$ and an irreducible component $F'$ of a fiber of $\psi$ intersect nontrivially}, we have $$\dim X \geq \dim F+\dim F'.$$
\end{lemma}

\begin{proof} The first part is known as Serre's inequality. Regarding the second part, if $\dim X < \dim F+\dim F'$, then the first part tells us the existence of a curve $C \subset F \cap F'$. However, this is a contradiction because the elementary contractions $\varphi$ and $\psi$ are distinct.   
\end{proof}

\begin{lemma}[{\cite[Lemma~2.5]{BCDD03}}]\label{lem:Fano:bundle} 
Let $X$ be a smooth Fano variety and $\psi: X \to Z$ be a $\P^r$-bundle. Then the following statements hold:
\begin{enumerate}
    \item The base space $Z$ is a smooth Fano variety with $\iota_Z \geq \iota_X$;
    \item If $\iota_X=\iota_Z$, then for a rational curve $f: \P^1\to Z$ such that $\deg f^{\ast}(-K_Z)=\iota_Z$, the fiber product $\P^1\times_Z X$ is isomorphic to $\P^1 \times \P^r$.
\end{enumerate}
\end{lemma}

\subsection{Families of rational curves}
Let $X$ be a smooth projective variety, and consider the space of rational curves denoted by ${\rm RatCurves}^n(X)$ (for details, see \cite[Section~II.2]{Kb}). A \textit{family of rational curves} $\mathcal{M}$ on $X$ is an irreducible component of $\text{RatCurves}^n(X)$. This family $\mathcal{M}$ comes with a $\mathbb{P}^1$-bundle $p: \mathcal{U} \to \mathcal{M}$ and an evaluation morphism $q: \mathcal{U} \to X$. We denote by $\text{Locus}(\mathcal{M})$ the union of all curves parameterized by $\mathcal{M}$. For a point $x \in X$, we denote by $\mathcal{M}_x$ the normalization of $p(q^{-1}(x))$ and by $\text{Locus}(\mathcal{M}_x)$ the union of all curves parametrized by $\mathcal{M}_x$. An {\it $\cM$-curve} is a curve parametrized by $\cM$. {Any $\mathcal{M}$-curve is numerically equivalent to any other, so 
the intersection number $(-K_X) \cdot C$ does not depend on the choice of an $\mathcal{M}$-curve $C$. 
We therefore define the anticanonical degree 
$\deg_{(-K_X)} \mathcal{M}$ of $\mathcal{M}$ as $(-K_X) \cdot C$ 
for $[C] \in \mathcal{M}$.
}

\begin{definition} Let $X$ be a smooth projective variety and $\mathcal{M} \subset {\rm RatCurves}^n(X)$ a family of rational curves. 
\begin{enumerate}
\item $\cM$ is a \textit{dominating family} (resp. \textit{covering family}) if $\overline{\text{Locus}(\mathcal{M})}=X$ (resp. ${\text{Locus}(\mathcal{M})}=X$).
\item ${\cM}$ is a \textit{minimal dominating family} if it is a dominating family with the minimal degree for some fixed ample line bundle $A$ on $X$. When $X$ is a {smooth} Fano variety, we always take the anticanonical divisor $-K_X$ as $A$ in this paper. 
\item $\mathcal{M}$ is \textit{locally unsplit} if, for a general point $x \in \text{Locus}(\mathcal{M})$, the family $\mathcal{M}_x$ is proper.
\item $\mathcal{M}$ is \textit{unsplit} if the family $\mathcal{M}$ is proper.
\end{enumerate}
\end{definition}

A minimal dominating family $\cM$ is locally unsplit. 
For an unsplit family $\mathcal{M}$ and a closed subset $Y \subset X$ such that $\text{Locus}(\mathcal{M})\cap Y$ is nonempty, we denote by $\text{Locus}(\mathcal{M})_Y$ the set of points that can be connected to $Y$ by a {$\cM$-curve}.

\begin{theorem}\label{them:fano:min:fam} 
Let $X$ be an $n$-dimensional smooth Fano variety, and let $\cM$ be a locally unsplit family of rational curves on $X$. Then the following statements hold:
\begin{enumerate}
    \item $\deg_{(-K_X)}\cM \leq n+1$.
    \item If $\deg_{(-K_X)}\cM= n+1$, then $X$ is isomorphic to $\mathbb{P}^n$.
    \item {If $\deg_{(-K_X )} \mathcal M=n$, then $X$ is isomorphic either to $Q^n$ or
to the blow-up of $\mathbb P^n$ along a smooth codimension two subvariety
contained in a hyperplane. Moreover, in the latter case, we have $\iota_X= 1$.}
\end{enumerate}
\end{theorem}

\begin{proof} The first and second statements follow from \cite{CMSB} (see also \cite{Keb02}). {The former part of the third statement follows from  \cite{DH17} and \cite{CD15}. To check the latter part of the third statement} on the pseudoindex, let $\varphi: X:=Bl_W(\P^n)\to \P^n$ be the blow-up of $\P^n$ along a codimension two smooth subvariety $W$ contained in a hyperplane, and let $E$ be the exceptional divisor and $\ell$ be a fiber of $\varphi|_E: E \to \varphi(E)$. Then we have $-K_X\cdot \ell=\varphi^*(-K_{\P^n})\cdot \ell -E\cdot \ell=-E\cdot \ell=1$.
\end{proof}

\begin{proposition}[{\cite[IV Corollary~2.6]{Kb}}]\label{prop:Ion:Wis:2}  
Let $X$ be a smooth projective variety, and let $\mathcal{M}$ be a locally unsplit family of rational curves on $X$. For a general point $x \in {\rm Locus}(\mathcal{M})$,
\[
\dim {\rm Locus}(\mathcal{M}_x) \geq \deg_{(-K_X)} \mathcal{M} + {\codim}_X {\rm Locus}(\mathcal{M}) - 1.
\]
Moreover, if $\mathcal{M}$ is unsplit, this inequality holds for any point $x \in {\rm Locus}(\mathcal{M})$.
\end{proposition}

\begin{proposition}[{\cite[Lemma~5.4]{ACO04}}]\label{prop:LocMY:dim} Let $X$ be a smooth projective variety and $Y$ an irreducible closed subset, and let $\mathcal{M}$ be an unsplit family of rational curves on $X$. Assume that curves in $Y$ are numerically independent of curves in $\cM$ and $Y \cap {\rm Locus}(\cM)$ is nonempty. Then we have
$$
\dim {\rm Locus}(\cM)_Y \geq \dim Y+\deg_{(-K_X)}\cM-1.
$$
\end{proposition}

\begin{remark}\label{rem:suzuki} {Although Proposition 2.10 (originally stated as {\cite[Lemma~5.4]{ACO04}}) does not require $Y$ to be irreducible, this condition is in fact
necessary. To see this, assume that $Y$ consists of two irreducible
components, $Y_1$ and $Y_2$, where only $Y_1$ meets ${\rm Locus}(\mathcal M)$. In this
case, increasing the dimension of $Y_2$ has no effect on the left-hand
side but arbitrarily enlarges the right-hand side of the inequality, thus invalidating the estimate. }
\end{remark}

\begin{lemma}[{\cite[Lemma~3.2]{Occ06}}]\label{lem:LocMY} Let $X$ be a smooth projective variety and $Y$ a closed subset, and let $\mathcal{M}$ be an unsplit family of rational curves on $X$. If ${\rm Locus}(\cM)\cap Y$ is nonempty, then {${\rm Locus}(\cM)_Y$} is closed in $X$. Moreover, for any curve $C\subset {\rm Locus}(\cM)_Y$, there exist a non-negative rational number $a$, a rational number $b$, and a curve $C_Y$ contained in $Y$ and an $\cM$-curve $C_{\cM}$ that satisfy the following
$$
[C]= a[C_Y]+b[C_{\cM}]\in N_1(X).
$$
\end{lemma}

\begin{lemma}[{\cite[Lemma~2.7]{Wat21}}]\label{lem:degeneration:curves} Let $X$ be a smooth projective variety and $\cM$ a family of rational curves on $X$. If $\cM$ is not unsplit, then there exists a rational $1$-cycle $Z=\sum_{i=1}^s a_i Z_i$ satisfying the following conditions:
\begin{enumerate}
\item $Z$ is algebraically equivalent to $\cM$-curves, where each $a_i$ is a positive integer and each $Z_i$ is a rational curve;
\item $\sum_{i=1}^s a_i\geq 2$.
\end{enumerate}
\end{lemma}

\subsection{Specific calculation of pseudoindex} In this subsection, we explicitly calculate the pseudoindex of the {smooth} Fano varieties appearing in Theorem~\ref{them:main}. Although this result may be known to experts, we present it for the convenience of the reader.

\begin{example}\label{eg:ps1} 
Let $X$ be one of the smooth Fano varieties as in Theorem~\ref{them:main}, other than cases (i), (iv), and (viii). We claim that the Fano index and the pseudoindex of $X$ are equal to $\frac{n}{2}$. By Theorem~\ref{them:Wis}, it is sufficient to show that $-K_X$ is divisible by $\frac{n}{2}$ in ${\rm Pic}X$. 

Let $X$ be the blow-up of $Q^n$ along a linear subspace $\P^{{\frac{n}{2} - 1}}$ of $Q^n$, that is, $$\varphi: X:=Bl_{\P^{{\frac{n}{2} - 1}}}(Q^n) \to Q^n.$$ By the canonical bundle formula for blow-ups, we have $$-K_X=\frac{n}{2}\left(2\varphi^{\ast}H-E\right),$$ where $H$ is a hyperplane section of $Q^n$ and $E$ is the exceptional divisor of $\varphi$. {Thus $-K_X$ is divisible by $\frac{n}{2}$ in ${\rm Pic}X$.} The same holds true when $X$ is the variety as in (iii) of Theorem~\ref{them:main}. Similarly, for varieties in cases (v)-(vii) or (iv) of Theorem~\ref{them:main}, the canonical bundle formula for projective bundles gives $\iota_X={ \rm i_X}=\frac{n}{2}$. 

{For the variety $X$ in (viii) of Theorem~\ref{them:main}, we see that $\iota_X=2$ and $ \rm i_X=1$.}
\end{example}

\begin{example}\label{eg:ps2} Let $X$ be the blow-up of $\P^n$ along a linear subspace $\P^t$ with $t \leq \frac{n}{2}-1$, that is, 
$$
X=Bl_{\P^t}(\P^n)=\P(\cO_{\P^{n-t-1}}(1)\oplus\cO_{\P^{n-t-1}}^{\oplus t+1}).
$$
Following the idea as in the proof of \cite[Theorem~1]{Tsu12}, we compute the pseudoindex of $X$. 
We denote the blow-up morphism by $$\varphi: X=Bl_{\P^t}(\P^n) \to \P^n.$$ We denote the $\P^{t+1}$-bundle morphism by $$\psi: X=\P(\cO_{\P^{n-t-1}}(1)\oplus\cO_{\P^{n-t-1}}^{\oplus t+1}) \to \P^{n-t-1}.$$ Let $E$ be the exceptional divisor of $\varphi$; then it can be written as 
$$
E=\P({{N^\vee_{\P^t/\P^n}}})=\P(\cO_{\P^t}(-1)^{\oplus n-t})=\P^{n-t-1}\times \P^t.
$$
We denote by $\cO_E(1)$ the tautological line bundle of $E=\P({{N^\vee_{\P^t/\P^n}}})$. Then we have $\cO_E(E)=\cO_E(-1)$. 
For a line $\ell \subset \P^t$, set $C':=\P(\cO_{\ell}(-1))\subset E$, which is an extremal rational curve of $\psi$. Then we have 
$$
E\cdot C'=\cO_E(-1)\cdot C'=(-1)\cdot (-1)=1.
$$
We denote by $C$ a line in a fiber of $\varphi|_E: E \to \P^t$. This is an extremal rational curve of $\varphi$. Then we have 
\begin{eqnarray}\label{eq:ext:curve}
-K_X\cdot C=n-t-1,\quad -K_X\cdot C'=t+2.
\end{eqnarray}
For any curve $\Gamma \subset X$, we claim 
\begin{eqnarray}\label{eq:-KX}
-K_X\cdot \Gamma \geq \min\{n-t-1, t+2\}.
\end{eqnarray}
We will prove this by dividing it into two cases. 

Firstly, assume that $\Gamma$ is not contained in $E$. In this case, we may take a hyperplane $H\subset \P^n$ such that $\P^t \subset H$ and $\varphi(\Gamma) \not\subset H$. For the strict transform $\tilde{H}$ of $H$, we have $\varphi^{\ast}H\cdot \Gamma=(E+\tilde{H})\cdot \Gamma \geq E\cdot \Gamma$. This yields
$$
-K_X\cdot \Gamma=\left\{(n+1)\varphi^{\ast}H-(n-t-1)E \right\}\cdot \Gamma \geq (t+2)E\cdot\Gamma\geq t+2.
$$

Secondly, assume that $\Gamma$ is contained in $E\cong \P^{n-t-1}\times \P^t$. In this case, there exist non-negative integers $a$ and $b$ such that $[\Gamma]= a[C]+a[C'] \in N_1(X)$ and $(a,b)\neq (0,0)$. By the equalities (\ref{eq:ext:curve}), we have 
$$
-K_X\cdot \Gamma=a(n-t-1)+b(t+2)\geq \min \{n-t-1, t+2\}.
$$
Thus, our claim holds. Combining the inequality (\ref{eq:-KX}) with the equalities (\ref{eq:ext:curve}), we obtain $\iota_X=\min\{n-t-1, t+2\}.$ This implies that $\iota_X=\frac{n}{2}$ if and only if $t=\frac{n}{2}-1$ or $\frac{n}{2}-2$. In particular, the pseudoindex of {smooth} Fano varieties as in (i) and (iv) of Theorem~\ref{them:main} are both $\frac{n}{2}$. 

{{If $X$ is (i) or (iv), then, by the canonical bundle formula for projective bundles, we have ${\rm i}_X=1$.}}
\end{example}

\section{Special cases}

\begin{lemma}\label{lem:n=4} Let $X$ be a $4$-dimensional smooth Fano variety. We assume that $\iota_X=2$ and $\rho_X>1$. If $X$ admits an elementary birational contraction, then $X$ is isomorphic to one of the varieties described in Theorem~\ref{them:main}.
\end{lemma}

\begin{proof} This is a direct consequence of {\cite[Theorem~6.2]{Wat24}} and \cite{Wis90}.
\end{proof}

\begin{lemma}\label{lem:pic3} Let $X$ be an $n$-dimensional smooth Fano variety. We assume that $\iota_X=\frac{n}{2}\geq 3$ and $\rho_X\geq 3$. Then $X$ is isomorphic to $(\P^2)^3$.
\end{lemma}

\begin{proof} We have $\iota_X=\frac{n}{2}\geq \frac{n+2}{3}$. Then, our assertion follows from \cite[Theorem]{Nov16}.
\end{proof}

To prove Theorem~\ref{them:main}, by using Lemma~\ref{lem:n=4} and Lemma~\ref{lem:pic3}, we need only to address the case where $n=\dim X \geq 6$ and $\rho_X=2$. As we will see in Theorem~\ref{them:main1} and Lemma~\ref{lem:dimF:l(R)} below, under the assumptions and notation of Theorem~\ref{them:main}, $X$ admits either a blow-up structure or a projective bundle structure. Therefore, the remainder of this section deals with those cases.

 \begin{theorem}[{\cite[Theorem~1.3]{AO05}}]\label{them:AO:blup} Let $X$ be an $n$-dimensional smooth Fano variety and let $R$ be an extremal ray of $\overline{NE}(X)$. Assume that the associated contraction $\varphi_R : X \to Y$ is the blow-up of a smooth variety $Y$ along a smooth subvariety $T \subset Y$ such that $\iota_X + \ell(R) \geq n$ or equivalently $\iota_X\geq \dim T+1$.
Then $X$ is isomorphic to one of the following:
\begin{enumerate}
    \item $Bl_{\P^t}(\P^n)$, with $\P^t$ a linear subspace of dimension $\leq \frac{n}{2} - 1$,
    \item $Bl_{\P^t}(Q^n)$, with $\P^t$ a linear subspace of dimension $\leq \frac{n}{2} - 1$,
    \item $Bl_{Q^t}(Q^n)$, with $Q^t$ a smooth quadric of dimension $\leq \frac{n}{2} - 1$ not contained in a linear subspace of $Q^n$,
    \item {$Bl_{Y\sqcup\{p\}}(\P^n)$ where $Y$ is a smooth subvariety of dimension $n-2$ and degree $\leq n$ contained in a hyperplane $H\subset \P^n$ such that $p \notin H$,}
    \item $Bl_{\P^1 \times \{p\}}(\P^1 \times \P^{n-1})$.
\end{enumerate}
\end{theorem}

\begin{proposition}\label{prop:bundle} Let $X$ be an $n$-dimensional smooth Fano variety admitting a $\P^{\iota_X-1}$-bundle structure $\psi: X\to \P^{\iota_X+1}$, and assume that $\iota_X=\frac{n}{2}\geq 2$. Then $X$ is isomorphic to one of the following:
\begin{enumerate} 
\item $\P^{\iota_X-1}\times \P^{\iota_X+1}$;
\item $\P(\cO_{\P^{\iota_X+1}}(1)\oplus\cO_{\P^{\iota_X+1}}^{\oplus \iota_X-1})$;
\item $\P(\cO_{\P^{\iota_X+1}}(1)^{\oplus 2}\oplus\cO_{\P^{\iota_X+1}}^{\oplus \iota_X-2})$;
\item $\P(\cO_{\P^{\iota_X+1}}(2)\oplus\cO_{\P^{\iota_X+1}}^{\oplus \iota_X-1}).$
\end{enumerate}
\end{proposition}

\begin{proof} Since the Brauer group of $\P^{\iota_X+1}$ vanishes, there exists a rank $\iota_X$ vector bundle $\cE$ over $\P^{\iota_X+1}$ such that $\psi: X \to \P^{\iota_X+1}$ is given by the natural projection $\P(\cE) \to \P^{\iota_X+1}$. Let $\ell\subset \P^{\iota_X+1}$ be a line. By tensoring a line bundle if necessary, we may assume that 
$$\cE|_{\ell}\cong \cO_{\P^1}(-a_1)\oplus \cO_{\P^1}(-a_2)\oplus \cdots\oplus  \cO_{\P^1}(-a_{\iota_X})
$$
with $0=a_1\leq  a_2\leq \cdots\leq a_{\iota_X}$.

Set $X_{\ell}:=\ell\times_{\P^{\iota_X+1}}X$ and $\tilde{\ell}:=\P(\cO_{\P^1}(-a_{\iota_X}))$. Remark that $\tilde{\ell}$ is a section of $X_{\ell}\to \ell$. Then we have 
\begin{align} 
\iota_X &\leq -K_X\cdot \tilde{\ell}= \psi^{\ast}(-K_{\P^{\iota_X+1}})\cdot\tilde{\ell}-K_{X/\P^{\iota_X+1}}\cdot\tilde{\ell}\nonumber\\
&= \left(\iota_X+2\right)+\left(-K_{X_{\ell}/\ell}\cdot \tilde{\ell}\right)= \left(\iota_X+2\right)-\iota_Xa_{\iota_X}+\sum_{i=1}^{\iota_X}a_i. \nonumber
\end{align}
This implies that 
$$
0\leq (a_{\iota_X}-a_1)+(a_{\iota_X}-a_2)+\cdots +(a_{\iota_X}-a_{\iota_X})\leq 2.
$$
Therefore, $\cE|_{\ell}$ is isomorphic to 
$$
\cO_{\P^1}^{\oplus \iota_X}, \quad \cO_{\P^1}\oplus\cO_{\P^1}(-1)^{\oplus \iota_X-1},  \quad \cO_{\P^1}^{\oplus2
}\oplus \cO_{\P^1}(-1)^{\oplus \iota_X-2}\quad \mbox{or} \quad \cO_{\P^1}\oplus \cO_{\P^1}(-2)^{\oplus \iota_X-1}.
$$
Since $\deg \cE|_{\ell}$ is independent of the choice of a line $\ell \subset \P^{\iota_X+1}$, we see that $\cE$ is a uniform vector bundle.
Applying \cite[Main Theorem]{Sato76}, $\cE$ is isomorphic to one of the following:
\begin{itemize}
\item $\cO_{\P^{\iota_X+1}}^{\oplus \iota_X}$;
\item $\cO_{\P^{\iota_X+1}}\oplus\cO_{\P^{\iota_X+1}}(-1)^{\oplus \iota_X-1}$;
\item $\cO_{\P^{\iota_X+1}}^{\oplus2}\oplus \cO_{\P^{\iota_X+1}}(-1)^{\oplus \iota_X-2}$;
\item $\cO_{\P^{\iota_X+1}}\oplus \cO_{\P^{\iota_X+1}}(-2)^{\oplus \iota_X-1}$.
\end{itemize}

\end{proof}

\begin{proposition}\label{prop:bundle:Q} Let $X$ be an $n$-dimensional smooth Fano variety. Assume that $X$ admits a $\P^{\iota_X-1}$-bundle structure $\psi: X\to Q^{\iota_X+1}$ and $\iota_X=\frac{n}{2}\geq 2$. Then $X$ is isomorphic to one of the following:
\begin{enumerate} 
\item $\P^{\iota_X-1}\times Q^{\iota_X+1}$;
\item $\P(\cO_{Q^{\iota_X+1}}(1)\oplus\cO_{Q^{\iota_X+1}}^{\oplus \iota_X-1})$;
\item $\P(\cS(1)\oplus \cO_{Q^4})$, where $\cS$ is the spinor bundle on $Q^4$ \cite{Ott88}.
\end{enumerate}
\end{proposition}

\begin{proof} By the same argument as in Proposition~\ref{prop:bundle}, there exists a vector bundle $\cE$ over $Q^{\iota_X+1}$ such that 
\begin{itemize}
\item $\psi: X \to Q^{\iota_X+1}$ is given by the natural projection $\P(\cE) \to Q^{\iota_X+1}$;
\item For any line $\ell \subset Q^{\iota_X+1}$, $\cE|_{\ell}$ is isomorphic to $\cO_{\P^1}^{\oplus \iota_X}$ or $\cO_{\P^1}\oplus\cO_{\P^1}(-1)^{\oplus \iota_X-1}$.
\item $\cE$ is a uniform vector bundle.
\end{itemize}
If $\cE|_{\ell}$ is isomorphic to $\cO_{\P^1}^{\oplus \iota_X}$ for any line $\ell \subset Q^{\iota_X+1}$, \cite[Proposition~1.2]{AW01} implies $\cE\cong \cO_{Q^{\iota_X+1}}^{\oplus \iota_X}$; thus $X$ is as in (i). 
So, we assume that $\cE|_{\ell}$ is isomorphic to $\cO_{\P^1}\oplus\cO_{\P^1}(-1)^{\oplus \iota_X-1}$ for any line $\ell \subset Q^{\iota_X+1}$. Then we have $\det \cE\cong \cO_{Q^{\iota_X+1}}(1-\iota_X)$; this implies that $-K_X=\iota_X\xi +n\psi^{\ast}H$, where $\xi$ is the tautological divisor of $X=\P(\cE)$ and $H$ is a hyperplane of $Q^{\iota_X+1}$. Set $L:=\xi +2\psi^{\ast}H$ and $\cE':=\psi_{\ast}\cO(L)$. {Then $\cE'=\cE(2)$ and this implies that the anticanonical divisor of ${Q^{\iota_X+1}}$ is equal to $\det \cE'$.} Applying \cite{PSW92} and \cite{Occ05}, we see that $X$ is isomorphic to the variety as in (ii) or (iii).
\end{proof}

\begin{remark}\label{rem:spinor} The vareity $\P(\cS(1)\oplus \cO_{Q^4})$ as in Proposition~\ref{prop:bundle:Q} has a scroll structure with a $4$-dimensional fiber (see \cite[Remarks~4.2]{Wis94}). In particular, it does not admit a birational contraction.
\end{remark}

\begin{proposition}\label{prop:bundle:iotaZ} Let $X$ be an $n$-dimensional smooth Fano variety. Assume that {there exists a vector bundle $\cE$ over a variety $Z$ such that  $X$ is isomorphic to $\P(\cE)$.} Moreover, we assume that $\iota_X=\iota_Z=\frac{n}{2}\geq 2$ and $\rho_X=2$. Then $X$ is isomorphic to $\P^{\iota_X-1}\times Z$.
\end{proposition}

\begin{proof} Let us take a minimal dominating family $\cN\subset \text{RatCurves}^n(Z)$. By Theorem~\ref{them:fano:min:fam} and our assumption, we obtain 
$$
\iota_Z\leq \deg_{(-K_Z)}\cN<\dim Z=\iota_X+1=\iota_Z+1.
$$ 
Thus we have $\deg_{(-K_Z)}\cN=\iota_Z$; {this implies that $\cN$ is an unsplit covering family. In fact, if ${\mathcal M}$ is not unsplit, we obtain a rational $1$-cycle $\sum_{i=1}^s a_iZ_i$ as in Lemma~\ref{lem:degeneration:curves}. Then we have $\deg_{(-K_X)}{\mathcal M}=\sum_{i=1}^s a_i(-K_X)\cdot Z_i\geq 2\iota_X=n$. This is a contradiction.} 
For $[\ell] \in \cN$, let $f:\P^1\to \ell \subset Z$ be the normalization of $\ell$; Lemma~\ref{lem:Fano:bundle} (ii) implies that $\P(f^{\ast}\cE)=\P^1\times_ZX$ is isomorphic to $\P^1\times \P^{\iota_X-1}$. This means that $\cE|_{\ell} \cong \cO_{\P^1}(a)^{\oplus \iota_X}$ for some integer $a$. Since $\deg \cE|_{\ell}$ is independent of the choice of $[\ell] \in \cN$, we see that $\cE$ is a uniform vector bundle with respect to $\cN$. Thus \cite[Proposition~1.2]{AW01} concludes that $X$ is isomorphic to $\P^{\iota_X-1}\times Z$.
\end{proof}

\section{Proof of the main theorem} 

\subsection{The case where $X$ admits an elementary small contraction} 
The purpose of this subsection is to prove the following:

\begin{theorem}\label{them:main1} Let $X$ be a smooth Fano variety with $\iota_X=\frac{n}{2}\geq 3$ and $\rho_X=2$. Assume that $X$ admits a small contraction $\varphi: X\to Y$ of an extremal ray $R\subset \overline{NE}(X)$. Then $X$ is isomorphic to $\P(\cO_{\P^{\iota_X+1}}^{\oplus 2}(1)\oplus\cO_{\P^{\iota_X+1}}^{\oplus \iota_X-2})$.
\end{theorem}

\begin{proof} Let $\cM\subset \text{RatCurves}^n(X)$ be a minimal dominating family of rational curves. Since $\rho_X>1$ and $\iota_X=\frac{n}{2}$, by Theorem~\ref{them:fano:min:fam}, we have $\deg_{(-K_X)}\cM<n$. Applying Lemma~\ref{lem:degeneration:curves}, it follows that $\cM$ is unsplit. By Proposition~\ref{prop:Ion:Wis:2}, for any $x\in X$, the following inequality holds
$$
\dim \text{Locus}(\cM_x) \geq \deg_{(-K_X)}\cM -1\geq \iota_X-1.
$$Let $F$ be an irreducible component of a nontrivial fiber of $\varphi$, and let $E$ be an {irreducible component of} the exceptional locus of $\varphi$ {containing $F$}. By Proposition~\ref{prop:Ion:Wis}, we have an inequality
\begin{eqnarray}\label{ineq2}
\dim F \geq \ell(R)+1\geq \iota_X+1.
\end{eqnarray}
{Applying} Proposition~\ref{prop:LocMY:dim}, we obtain 
$$
n \geq \dim \text{Locus}(\cM)_F \geq \dim F+\deg_{(-K_X)}\cM -1 \geq \left( \iota_X+1 \right)+\iota_X-1=n.
$$ 
This inequality implies that $\dim F = \iota_X + 1$, $X = \mathrm{Locus}(M)_F$, and {$\deg_{(-K_X)}\cM = \iota_X$.} Then Proposition~\ref{prop:Ion:Wis} implies that $\mathrm{codim}\, E = 2$ and $\iota_X = \ell(R)$.
 
Since $\rho_X=2$, there exists an extremal ray $R'\subset \overline{NE}(X)$ which is not $R$. Let $C$ and $C'$ denote an extremal rational curve of $R$ and of $R'$, respectively. We claim that $R'$ contains any $\cM$-curve. {Since any curve contained in $F$ is numerically proportional to $C$, Lemma~\ref{lem:LocMY} implies} there exist a non-negative rational number $a$, a rational number $b$ and an $\cM$-curve $C_{\cM}$ that satisfy the following
$$
[C']= a[C]+b[C_{\cM}]\in N_1(X).
$$
If $b$ is negative, we have $[C'], [C_{\cM}] \in R$ since $[C']-b[C_{\cM}]= a[C]\in R$. This is a contradiction. Hence, we have $b$ is non-negative. Then $a[C]+b[C_{\cM}]=[C']\in R'$; this yields $a[C], b[C_{\cM}]\in R'$. Since $R$ and $R'$ are distinct, we obtain $a=0$. Then we see that $b\neq 0$ and $[C_{\cM}]\in R'$. Since any $\cM$-curve is numerically equivalent to each other, our claim holds.

Let $\psi: X \to Z$ be the contraction of $R'$, which contracts any $\cM$-curve. Since we have a finite morphism $\psi|_F: F\to Z$, we have $\dim Z \geq \dim F=\iota_X+1$. We also have an inequality 
$$
\dim X\geq \dim \text{Locus}(\cM_x) + \dim Z \geq \iota_X-1+\dim Z.
$$
This yields that $\iota_X+1\geq \dim Z$. Thus we see that $\dim Z=\iota_X+1$, and this concludes that $\psi|_F: F\to Z$ is surjective and the dimension of a general fiber of $\psi$ is $\iota_X-1$. 

Here, we claim that $\psi: X\to Z$ is equidimensional. Suppose the contrary to prove this. Then { there exists an irreducible component $F'$ of a jumping fiber of} $\psi$, such as $\dim F'>\iota_X-1$. {Since we have $\text{Locus}(\cM)_F=X$, for a general point $x\in F'$, we may find an $\cM$-curve $C$ such that $x \in C$ and $F \cap C \neq \emptyset$. Since $x$ is a general point of $F'$, $C$ is contained in $F'$. } Thus we see that $F \cap F'$ is nonempty. {Then, by Lemma~\ref{lem:Serre}, we have 
$$
\dim X \geq \dim F +\dim F' >\left(\iota_X+1\right)+\left(\iota_X-1\right)=\dim X.
$$}
This is a contradiction. Consequently, we see that $\psi: X\to Z$ is equidimensional.
Remark that $\iota_X=\deg_{(-K_X)}\cM$; this yields that $\ell(R')=\iota_X$. Applying Theorem~\ref{them:HN13}, we see that $\psi$ is a $\P^{\iota_X-1}$-bundle.

{
Let $\tilde{E}\to E$ be the normalization of $E$ and $\tilde{\varphi}: \tilde{E} \to \tilde{W}$ be the morphism obtained by the Stein factorization of $\tilde{E} \to E \to {Y}$.
Since we have $\dim E+\dim F=\dim X + \ell(R)-1$, \cite[Theorem~1.4]{HNov13} implies $\tilde{E} \to \tilde{W}$ is a projective bundle in codimension one. From this, we obtain a surjective morphism ${\P^{\iota_X+1}\to Z}$. Applying \cite[Theorem~4.1]{Laz84}, it follows that $Z$ is isomorphic to $\P^{\iota_X+1}$. Then our assertion follows from Proposition~\ref{prop:bundle}.}

\end{proof}

\subsection{The case where $X$ admits an elementary divisorial contraction} 
The purpose of this subsection is to prove the following:
\begin{theorem}\label{them:main2} Let $X$ be a smooth Fano variety with $\iota_X=\frac{n}{2}\geq 3$ and $\rho_X=2$. Assume that $X$ admits a divisorial contraction $\varphi: X\to Y$ of an extremal ray $R\subset \overline{NE}(X)$. { Then $X$ is isomorphic to one of the varieties in Theorem~\ref{them:main}.}

\end{theorem}

Since $\rho_X=2$, there exists an extremal ray $R'\subset \overline{NE}(X)$ such that $\overline{NE}(X)= R+R'$. We denote by $C$ and $C'$ an extremal rational curve of $R$ and that of $R'$, respectively. Let $\psi:X \to Z$ be the contraction of $R'$.
We denote by $E$ the exceptional divisor of $\varphi$. Throughout this section, we will use these notations.
By the negativity lemma \cite[Lemma~3.39]{KM}, we have $E\cdot C<0$. Since $E$ is an effective divisor, a curve $C_0\subset X$ with $E\cdot C_0>0$ exists. For non-negative numbers $a, b$ such that $[C_0]=a[C]+b[C']\in N_1(X)$, we have 
$$
0<E\cdot C_0 =aE\cdot C+bE\cdot C'\leq bE\cdot C'.
$$
This implies that $E\cdot C'>0$. We now claim the following:
\begin{lemma} {The contraction} $\psi$ is of fiber type.
\end{lemma}  
\begin{proof} Assume the contrary, that is, $\psi$ is of birational type. If $\psi$ is a small contraction, then $X$ is isomorphic to the variety as in Theorem~\ref{them:main1}. In this case, $\psi$ is of fiber type; this is a contradiction. Thus, $\psi$ is a divisorial contraction. 
Let $E'$ be the exceptional divisor of $\psi$. We denote by $F$ and $F'$ an irreducible component of a nontrivial fiber of $\varphi$ and that of $\psi$, respectively. By Proposition~\ref{prop:Ion:Wis}, we have 
\begin{eqnarray}\label{eq:F}
\dim F \geq \ell(R)\geq \iota_X\quad\text{and}\quad \dim F'\geq \ell(R')\geq\iota_X.
\end{eqnarray}
Since $E\cdot C'>0$, $E\cap F'$ is nonempty. We claim that $\psi|_{E'}$ is an equidimensional morphism of relative dimension $\iota_X$. If not, there exists an irreducible component $F'_0$ of a nontrivial fiber of $\psi$ such that $\dim F'_0 >\iota_X$. Since $E\cap F'_0$ is nonempty, we may find an irreducible component $F_0$ of a nontrivial fiber of $\varphi$ such that $F_0 \cap F_0'$ is nonempty. Then we have 
$$
\dim F_0\cap F_0' \geq \dim F_0+\dim F_0'-n >0.
$$
This is a contradiction. Therefore, $\psi|_{E'}$ is an equidimensional morphism of relative dimension $\iota_X$.
By the inequality (\ref{eq:F}), we have $\ell(R')=\iota_X$. Applying Theorem~\ref{them:AO02}, we see that $\psi: X\to Z$ is the blow-up of a smooth projective variety $Z$ along a smooth variety of codimension $\iota_X+1$. Then $X$ is isomorphic to any of the varieties appearing in (i), (ii), and (iii) of Theorem~\ref{them:AO:blup}, but these varieties have only one birational contraction. Consequently, we see that $\psi$ is of fiber type.
\end{proof}

{Subsequently, an irreducible component of a non-trivial fibre of $\varphi$ and that of $\psi$ are denoted by $F$ and $F'$, respectively. Any fibre of $\psi$ whose dimension is equal to $\dim X - \dim Z$ is denoted by $F'_{\mathrm{gen}}$, and any fibre of $\varphi|_{E}$ whose dimension is equal to $\dim E - \dim \varphi(E)$ is denoted by $F_{\mathrm{gen}}$.}

Since $\psi|_F: F\to Z$ is finite, {applying Proposition~\ref{prop:Ion:Wis} for $\varphi$, we have} 
\begin{eqnarray}\label{ineq:F}
\dim Z\geq \dim F\geq\dim F_{\rm gen}\geq \ell(R)\geq \iota_X.
\end{eqnarray}
{Applying Proposition~\ref{prop:Ion:Wis} for $\psi$,} we obtain 
\begin{eqnarray}\label{ineq:F'}
\dim F'\geq \dim F_{\rm gen}'\geq \ell(R')-1\geq \iota_X-1.
\end{eqnarray}
Moreover, we have 
\begin{eqnarray}\label{ineq:X}
n=\dim F_{\rm gen}'+\dim Z\geq \left(\iota_X-1\right)+\iota_X=n-1.
\end{eqnarray}
This implies 
\begin{eqnarray}\label{ineq:F'Z}
\dim F_{\rm gen}'=\iota_X~\text{or}~\iota_X-1.
\end{eqnarray}

\begin{lemma}\label{lem:dimF:l(R)} {One of the following holds}:
\begin{enumerate} 
\item The contraction $\varphi:X \to Y$ is the blow-up of a smooth projective variety $Y$ along {a} smooth subvariety $W$ of dimension $\iota_X-1$; 
\item The contraction $\psi: X \to Z$ is a $\P^{\iota_X-1}$-bundle. 
\end{enumerate}
\end{lemma}

\begin{proof} {If $\dim F = \iota_X$ for any irreducible component $F$ of any fiber of $\varphi|_E$, then $\varphi|_E$ is equidimensional and $\dim F = \ell(R)$ by (\ref{ineq:F})}, so by Theorem~\ref{them:AO02} we are in case (i). Therefore, we may assume that there exists an irreducible component $F_0$ of a fiber of $\varphi|_E$ with $\dim F_0 \geq \iota_X + 1$.
Then, since $\psi|_{F_0} : F_0 \to Z$ is finite, it follows that
\[
\dim Z \geq \dim F_0 \geq \iota_X + 1. \tag{1}
\]
{Now, if $\dim F' = \iota_X - 1$ for any irreducible component $F'$ of any fiber of $\psi$, then $\psi$ is equidimensional and $\dim F' = \ell(R') - 1$ by (\ref{ineq:F'})}, so by Theorem~\ref{them:HN13}  we are in case (ii). Hence we may assume that there exists an irreducible component $F'_0$ of a fiber of $\psi$ with $\dim F'_0 \geq \iota_X$.
Let $\mathcal{M} \subset {\rm RatCurves}^n(X)$ be a dominating family of $\psi$-contracted rational curves of minimal anticanonical degree. By the same argument as in Theorem~\ref{them:main1}, we see that $\mathcal{M}$ is unsplit and that $\mathrm{Locus}(\mathcal{M})_{F_0} = X$. Moreover, by the same argument as in Theorem~\ref{them:main1}, we conclude that $F_0 \cap F'_0 \neq \emptyset$. 
{Then, by Lemma~\ref{lem:Serre}, we have $$\dim X \ge \dim F_0 + \dim F'_0 > (\iota_X+1) + \iota_X = \dim X+1,$$ a contradiction.}
\end{proof}

\begin{proof}[{Proof of Theorem~\ref{them:main2}}] We see that $\varphi$ and $\psi$ satisfy either (i) or (ii) of Lemma~\ref{lem:dimF:l(R)}.
In the first case, Theorem~\ref{them:AO:blup} yields that $X$ is isomorphic to one of the varieties as in (i)-(iii) of Theorem~\ref{them:main2}. So, assume that $\psi: X \to Z$ is a $\P^{\iota_X-1}$-bundle. {By Lemma~\ref{lem:Fano:bundle}, $Z$ is a smooth Fano variety with $\iota_Z \geq \iota_X$. Moreover, since $\dim Z=\dim X-(\iota_X-1)$, it follows that  $\iota_Z\leq \dim Z+1=\iota_X+2$.} If $\iota_Z=\iota_X+2$ or $\iota_X+1$, then Theorem~\ref{them:fano:min:fam} implies that $Z$ is isomorphic to $\P^{\iota_X+1}$ or $Q^{\iota_X+1}$, respectively. In these cases, by Proposition~\ref{prop:bundle} and Proposition~\ref{prop:bundle:Q}, $X$ is isomorphic to one of the varieties as in (iv)-(vi) of Theorem~\ref{them:main2}. Finally, Proposition~\ref{prop:bundle:iotaZ} implies that the case $\iota_Z=\iota_X$ does not occur. 
\end{proof}

\subsection{Conclusion} 
\begin{proof}[{Proof of Theorem~\ref{them:main}}] Our assertion follows from Example~\ref{eg:ps1}, Example~\ref{eg:ps2}, Lemma~\ref{lem:n=4}, Lemma~\ref{lem:pic3},  Theorem~\ref{them:main1}, and Theorem~\ref{them:main2}.
\end{proof}

\subsection*{Acknowledgments} The author would like to express sincere gratitude to the anonymous referee for carefully reading the manuscript and providing valuable comments that greatly improved this work.
The author is also deeply grateful to Professor Taku Suzuki for reading the manuscript and offering valuable suggestions; in particular, he kindly pointed out Remark \ref{rem:suzuki}.

Part of the results presented in this paper were obtained during the author's stay at Université Côte d'Azur. The author wishes to express his deepest gratitude to Professor Andreas Höring for hosting him during this period. The author is also profoundly grateful to Université Côte d'Azur for providing a stimulating and welcoming academic environment. The assistance and hospitality of the faculty and staff were greatly appreciated, making the author's stay both productive and memorable.

The author would like to express gratitude to the American Institute of Mathematics (AIM) and the participants of the AIM workshop {\it Higher-dimensional log Calabi-Yau pairs}, where topics related to this research were discussed. The discussions and feedback from the workshop significantly influenced and contributed to the development of this paper.

\bibliographystyle{plain} 
\bibliography{biblio}
\end{document}
